\documentclass[12pt,twoside,a4paper]{amsart}
\usepackage[centertags]{amsmath}
\usepackage{amssymb,amscd,amsfonts,latexsym,a4wide,amsthm}

\advance\headheight by 3pt

\renewcommand{\subjclass}[1]{\thanks{\emph{2010 Mathematics Subject Classification:}~#1}}
\renewcommand{\keywords}[1]{\thanks{\emph{Keywords and Phrases:}~#1}}
\renewcommand{\date}{\thanks{\today}}

\newtheorem{theorem}{Theorem}[section]
\newtheorem{lemma}{Lemma}[section]

\newtheorem{proposition}[lemma]{Proposition}
\numberwithin{equation}{section}

\def\teto#1{\setbox\z@\hbox{${#1\vphantom k}$}\hbox{%
 \hbox{\lower2\ex@\hbox{\lower\dp\z@\hbox{\vbox{\hrule
 \hbox{\vrule\hskip2\ex@\vbox{\vskip2\ex@\box\z@\vskip1\ex@}%
 \hskip2\ex@\vrule}}}}}}}
\catcode`\@=\active

\newcommand{\Q}{\mathbb{Q}}

\newcommand{\Z}{\mathbb{Z}}
\newcommand{\NN}{\mathbb{N}}

\newcommand{\Cc}{\mathbb{C}}
\newcommand{\Rr}{\mathbb{R}}
\newcommand{\Qq}{\mathbb{Q}}
\newcommand{\OQq}{\overline{\mathbb{Q}}}
\newcommand{\Zz}{\mathbb{Z}}

\newcommand{\m}{\mathbf{m}}

\newcommand{\cP}{\mathcal{P}}

\newcommand{\ve}{\varepsilon}
\newcommand{\al}{\alpha}

\newcommand{\ga}{\gamma}

\newcommand{\fp}{\frak{p}}
\newcommand{\fa}{\frak{a}}

\newcommand{\fd}{\frak{d}}

\newcommand{\fA}{\frak{A}}
\newcommand{\fB}{\frak{B}}
\newcommand{\fC}{\frak{C}}
\newcommand{\fD}{\frak{D}}
\newcommand{\fP}{\frak{P}}

\newcommand{\ketsor}[2]{\genfrac{}{}{0pt}{2}{#1}{#2}}

\newcommand{\half}{\mbox{$\frac{1}{2}$}}

\DeclareMathOperator{\ord}{ord}

\DeclareMathOperator{\rank}{rank}

\DeclareMathOperator{\lcm}{lcm}

\newcommand{\kdots}{,\ldots ,}
\renewcommand{\eqref}[1]{(\ref{#1})}

\title[Hyper- and superelliptic equations]{Effective results for hyper- and superelliptic equations over number fields}


\author[A. B\'erczes]{Attila B\'erczes}
\thanks{The research was supported in part by the Hungarian Academy of
Sciences, and by grants K100339 (A.B., K.G.)
and K75566 (A.B.) of the Hungarian National Foundation for
Scientific Research. The work is supported
by the T\'AMOP 4.2.1./B-09/1/KONV-2010-0007 project. The project
is implemented through the New Hungary Development Plan,
co-financed by the European Social Fund and the European Regional
Development Fund.}
\address{A. B\'erczes \newline
         \indent Institute of Mathematics, University of Debrecen \newline
         \indent H-4010 Debrecen, P.O. Box 12, Hungary}
\email{berczesa\char'100science.unideb.hu}

\author[J.-H. Evertse]{Jan-Hendrik Evertse}
\address{J.-H. Evertse \newline
         \indent Universiteit Leiden, Mathematisch Instituut, \newline
         \indent Postbus 9512, 2300 RA Leiden, The Netherlands}
\email{evertse\char'100math.leidenuniv.nl}

\author[K. Gy\H{o}ry]{K\'{a}lm\'{a}n Gy\H{o}ry}
\address{K. Gy\H{o}ry \newline
         \indent Institute of Mathematics, University of Debrecen \newline
         \indent H-4010 Debrecen, P.O. Box 12, Hungary}
\email{gyory\char'100science.unideb.hu}

\begin{document}
\maketitle

\centerline{\it "To the memory of Professor Antal Bege"}

\begin{abstract}
Let $f$ be a polynomial with coefficients in the ring $O_S$ of $S$-integers of a given
number field $K$, $b$ a non-zero $S$-integer, and $m$ an integer $\geq 2$.
Suppose that $f$ has no multiple zeros.
We consider the equation (*) $by^m=f(x)$ in $x,y\in O_S$.
In the present paper we give explicit upper bounds
in terms of $K,S,b,f,m$ for the heights of the solutions of (*).
Further, we give an explicit bound $C$ in terms of $K,S,b,f$
such that if $m>C$ then (*) has only solutions with $y=0$ or a root of unity.
Our results are more detailed versions of work of
Trelina, Brindza, and Shorey and Tijdeman.
The results in the present paper are needed in a forthcoming paper of ours
on Diophantine equations over integral domains which are finitely
generated over $\Zz$.
\end{abstract}

\maketitle

\section{Introduction}\label{introduction}

Let $f\in\Zz [X]$ be a polynomial of degree $n$ without multiple roots
and $m$ an integer $\geq 2$. Siegel proved that the equation
\begin{equation}\label{1.1}
y^m=f(x)
\end{equation}
has only finitely many solutions in $x,y\in\Zz$
if $m=2,n\geq 3$ \cite{Siegel26} and if $m\geq 3,n\geq 2$
\cite{Siegel29}. Siegel's proof is ineffective. In 1969, Baker \cite{Baker69}
gave an effective proof of Siegel's result. More precisely, he showed that
if $(x,y)$ is a solution of \eqref{1.1}, then
\[
\max(|x|,|y|)\leq\left\{
\begin{array}{ll}
\exp\exp\left\{ (5m)^{10}(n^{10n}H)^{n^2}\right\}&\mbox{if } m\geq 3,\, n\geq 2,
\\
\exp\exp\exp\left\{(10^{10n}H)^2\right\}&\mbox{if } m=2,\, n\geq 3,
\end{array}\right.
\]
where $H$ is the maximum of the absolute values of the coefficients of $f$.
In 1976, Schinzel and Tijdeman \cite{SchinzelTijdeman76} proved that there is
an effectively computable number $C$, depending only on $f$, such that \eqref{1.1}
has no solutions $x,y\in\Zz$ with $y\not= 0,\pm 1$ if $m>C$.
The proofs of Baker and of Schinzel and Tijdeman
are both based on Baker's results on linear forms in logarithms of
algebraic numbers.

First Trelina \cite{Trelina78} and later in a more general form Brindza \cite{B6}
generalized the results of Baker to equations of the type \eqref{1.1} where
the coefficients of $f$ belong to the ring of $S$-integers $O_S$
of a number field $K$
for some finite set of places $S$, and where the unknowns $x,y$
are taken from $O_S$. In their proof they used Baker's result on linear
forms in logarithms, as well as a $p$-adic analogue of this.
In fact, Baker, Schinzel and Tijdeman, Trelina and Brindza considered \eqref{1.1}
also for polynomials $f$ which may have multiple roots.
Brindza gave an effective bound for the solutions in the most general situation
where \eqref{1.1} has only finitely many solutions. This was later improved by
Bilu \cite{Bilu1} and Bugeaud \cite{Bug2}.
Shorey and Tijdeman \cite[Theorem 10.2]{ShoreyTijdeman86} extended the theorem
of Schinzel and Tijdeman to equation \eqref{1.1} over the $S$-integers of a number field.
For further related results and applications we refer to \cite{ShT}, \cite{Bilu1}, \cite{Bug2},
\cite{GyPA2} and the references given there.

In a forthcoming paper, we will prove effective analogues of the theorems
of Baker and Schinzel and Tijdeman for equations of the type \eqref{1.1}
where the unknowns $x,y$ are taken from an arbitrary finitely generated domain
over $\Zz$. For this, we need effective finiteness results for Eq. \eqref{1.1}
over the ring of $S$-integers of a number field which are more precise than the
results of Trelina, Brindza, Bilu, Bugeaud and Shorey and Tijdeman mentioned above.
In the present paper, we derive such precise results.
Here, we follow improved, updated versions of standard methods.
For technical convenience, we restrict ourselves
to the case that the polynomial $f$ has no multiple roots.
We mention that recently, Gallegos-Ruiz \cite{Gallegos-Ruiz11} obtained an explicit
bound for the heights of the solutions of the hyperelliptic equation $y^2=f(x)$
in $S$-integers $x,y$ over $\Q$, but his result is not adapted to our purposes.

In Theorems \ref{T_super} and \ref{T_hyper} stated below we give for any fixed
exponent $m$ effective upper bounds for the heights of the solutions $x,y\in O_S$
of \eqref{1.1} which are fully explicit in terms of $m$, the degree and height of $f$, the degree and
discriminant of $K$ and the prime ideals in $S$.
In Theorem \ref{T_ST} below we generalize the Schinzel-Tijdeman Theorem to the effect
that if \eqref{1.1} has a solution $x,y\in O_S$ with $y$ not equal to $0$
or to a root of unity, then $m$ is bounded above by an explicitly given bound
depending only on $n$, the height of $f$, the degree and
discriminant of $K$ and the prime ideals in $S$.

\section{Results}\label{S_Int}

We start with some notation.
Let $K$ be a number field. We denote by $d, D_K$ the degree and discriminant of $K$,
by $O_K$ the ring of integers of $K$ and by $M_K$ the set of places of $K$.
The set $M_K$ consists of
real infinite places, these are
the embeddings $\sigma :\, K\hookrightarrow \Rr$; complex infinite places,
these are the pairs of conjugate complex embeddings
$\{ \sigma ,\overline{\sigma}:\, K\hookrightarrow\Cc\}$,
and finite places, these are the prime ideals of $O_K$.
We define normalized absolute values $|\cdot |_v$ ($v\in M_K$) as follows:
\begin{equation}\label{abs_val}
\left\{
\begin{array}{l}
|\cdot|_v=|\sigma(\cdot)|\quad \mbox{if $v=\sigma$ is real infinite;}
\\
|\cdot|_v=|\sigma(\cdot)|^2\quad \mbox{if $v=\{\sigma,\, \overline{\sigma}\}$
is complex infinite;}
\\
|\cdot|_v=(N_K\fp)^{-\ord_{\fp}(\cdot)} \quad \mbox{if $v=\fp$ is finite;}
\end{array}\right.
\end{equation}
here $N_K\fp=\#O_K/\fp$ is the norm of $\fp$ and
$\ord_{\fp}(x)$ denotes the exponent of $\fp$ in the prime ideal decomposition
of $x$, with $\ord_{\fp}(0)=\infty$.

The logarithmic height of $\al\in K$ is defined by
\[
h(\al):=\frac{1}{[K:\Qq ]}\log\prod_{v\in M_K}\max (1,|\al |_v).
\]

Let $S$ be a finite set of places of $K$ containing all
(real and complex) infinite places.
We denote by $O_S$ the ring of $S$ integers in $K$, i.e.
$$
O_S=\{ x\in K \ : \ |x|_v \leq 1 \text{ for } v\in M_K\setminus S \}.
$$
Let $s:=\# S$ and put
\begin{eqnarray*}
&&P_S=Q_S:=1\ \ \mbox{if $S$ consists only of infinite places,}
\\
&&P_S=\max_{i=1\kdots t} N_K\fp_i,\ \ Q_S:=\prod_{i=1}^t N_K\fp_i
\\
&&\qquad\qquad\mbox{if $\fp_1\kdots\fp_t$ are the prime ideals in $S$.}
\end{eqnarray*}

We are now ready to state our results. In what follows,
\begin{equation}\label{poly}
f(X)=a_0X^n+a_1X^{n-1}+\dots+a_n \in O_S[X]
\end{equation}
is a polynomial of degree $n\geq 2$ without multiple roots and $b$ is a non-zero
element of $O_S$.
Put
$$
\widehat{h}:=\frac{1}{d}\sum_{v \in M_K}
\log \max (1,|b|_v,|a_0|_v, \dots ,|a_n|_v).
$$

Our first result concerns the superelliptic equation
\begin{equation}\label{super}
f(x)=by^m \quad \quad \mbox{in $x,y\in O_S$.}
\end{equation}
with a fixed exponent $m\geq 3$.

\begin{theorem}\label{T_super}
Assume that $m\geq 3$, $n\geq 2$.
If $x,y\in O_S$ is a solution to the equation \eqref{super} then we have
\begin{equation}\label{2.1}
h(x),h(y)\leq (6ns)^{14m^3n^3s}|D_K|^{2m^2n^2}Q_S^{3m^2n^2}e^{8m^2n^3d\widehat{h}}.
\end{equation}
\end{theorem}

We now consider the hyperelliptic equation
\begin{equation}\label{hyper}
f(x)=by^2\quad\quad \mbox{in $x,y\in O_S$.}
\end{equation}

\begin{theorem}\label{T_hyper}
Assume that $n\geq 3$.
If $x,y\in O_S$ is a solution to the equation \eqref{hyper} then we have
\begin{equation}\label{2.2}
h(x), h(y) \leq (4ns)^{212n^4s}|D_K|^{8n^3}Q_S^{20n^3}e^{50n^4d\widehat{h}}.
\end{equation}
\end{theorem}

Our last result is an
an explicit version of the Schinzel-Tijdeman theorem over the $S$-integers.

\begin{theorem}\label{T_ST}
Assume that \eqref{super} has a solution $x,y\in O_S$ where
$y$ is neither $0$ nor a root of unity. Then
\begin{equation}\label{bound_m}
m \leq (10n^2s)^{40ns}|D_K|^{6n}P_S^{n^2}e^{11nd\widehat{h}}.
\end{equation}
\end{theorem}

\section{Notation and auxiliary results}

We denote by $d,D_K,h_K,R_K$ the degree, discriminant, class number and regulator,
and by $O_K$ the ring of integers of $K$.
Further, we denote by $\cP(K)$ the collection of non-zero prime ideals of $O_K$.
For a non-zero fractional ideal $\fa$ of $O_K$ we have the unique factorization
$$
\fa=\prod_{\fp \in \cP(K)} \fp^{\ord_{\fp} \fa},
$$
where there are only finitely many prime ideals $\fp\in \cP(K)$ with $\ord_{\fp} \fa \ne 0$.
Given $\alpha_1\kdots\alpha_n\in K$, we denote by $[\alpha_1\kdots\alpha_n]_K$
the fractional ideal of $O_K$ generated by $\alpha_1\kdots \alpha_n$.
For a polynomial $f\in K[X]$ we denote by $[f]_K$ the fractional ideal
generated by the coefficients of $f$.
We denote by $N_K\fa$ the absolute norm of a fractional ideal of $O_K$.
In case that $\fa\subseteq O_K$ we have $N_K\fa =\# O_K/\fa$.

We define $\log^* x:=\max (1,\log x)$ for $x\geq 0$.

\subsection{Discriminant estimates}

Let $L$ be a finite extension of $K$.
Recall that the relative discriminant ideal
$\fd_{L/K}$ of $L/K$ is the ideal of $O_K$
generated by the numbers
$$
D_{L/K}(\omega_1, \dots , \omega_n)\ \ \mbox{with } \omega_1, \dots \omega_n \in O_L ,
$$
where $n:=[L:K]$.

\begin{lemma}\label{L_gen_disc_I}
Suppose that $L=K(\al)$ and let $f \in K[X]$ be
a square-free polynomial of degree $m$ with $f(\al)=0$.
Then
\begin{equation}\label{gen_disc_I}
\fd_{L/K} \ \supseteq \frac{[D(f)]_K}{[f]_K^{2m-2}}.
\end{equation}
\end{lemma}

\begin{proof}
We have inserted a proof for lack of a good reference.
We write $[\cdot ]$ for $[\cdot ]_K$.
Let $g\in K[X]$ be the monic minimal polynomial of $\al$. Then $f=g_1g_2$ with $g_2\in K[X]$. Let $n:=\deg g_1$ and $k:=\deg h_1$.
Then
$$
D(f)=D(g_1)D(g_2)R(g_1,g_2)^2,
$$
where $R(g_1,g_2)$ is the resultant of $g_1$ and $g_2$.
Using determinantal expressions for
$D(g_1)$, $D(g_2)$, $R(g_1,g_2)$ we get
$$
D(g_1)\in [g_1]^{2n-2}, \quad D(g_2)\in [g_2]^{2k-2}, \quad R(g_1,g_2) \in [g_1]^k[g_2]^n,
$$
and by Gauss' Lemma, $[f]=[g_1]\cdot [g_2]$. Hence
$$
\frac{[D(f)]}{[f]^{2m-2}}=\frac{[D(g_1)]}{[g_1]^{2n-2}}\frac{[D(g_2)]}{[g_2]^{2k-2}}\frac{[R(g_1,g_2)]}{[g_1]^k[g_2]^n} \subseteq \frac{[D(g_1)]}{[g_1]^{2n-2}}.
$$
Therefore, it suffices to prove
$$
\fd_{L/K}\supset\frac{[D(g_1)]}{[g_1]^{2n-2}}.
$$
Note that $[g_1]^{-1}$ consists of all $\lambda\in K$ with $\lambda g_1\in O_K[X]$.
Hence the ideal $[D(g_1)]\cdot [g_1]^{-2n+2}$
is generated by the numbers $\lambda^{2n-2}D(g_1)=D(\lambda g_1)$
such that $\lambda g_1\in O_K[X]$. Writing $h:=\lambda g_1$, we see that it
suffices to prove that if $h\in O_K[X]$ is irreducible in $K[X]$
and $h(\alpha )=0$ with $L=K(\alpha )$, then
$$
D(h)\in \fd_{L/K}.
$$
To prove this, we use an argument of  Birch and Merriman \cite{BirMer}.
Let $h(X)=b_0X^m+b_1x^{m-1}+\dots +b_m \in O_K[X]$ with $h(\al)=0$.
Put
\[
\omega_i:=b_0\al^i+b_1\al^{i-1}+\cdots +b_i\ \ (i=0,1, \dots, n).
\]
We show by induction on $i$ that $\omega_i\in O_L$.
For $i=0$ this is clear. Assume that we have
proved that $\omega_i\in O_L$ for some $i \geq 0$. By $h(\al)=0$ we clearly have
\[
\omega_i\al^{n-i}+b_{i+1}\al^{n-i-1}+\cdots+b_n=0.
\]
By multiplying this expression with $\omega_i^{n-i-1}$, we see that
$\omega_i\al$ is a zero of a monic polynomial from $O_L[X]$,
hence belongs to $O_L$.
Therefore, $\omega_{i+1}=\omega_i\al+b_{i+1} \in O_L$.

Now on the one hand, $D_{L/K}(1,\omega_1, \dots, \omega_{n-1})\in \fd_{L/K}$,
on the other hand,
\begin{eqnarray*}
D_{L/K}(1,\omega_1, \dots, \omega_{n-1})&=&b_0^{2n-2}D_{L/K}(1,\al,\dots, \al^{n-1})
\\
&=&
b_0^{2n-2}\prod_{1\leq i<j\leq 0}(\al^{(i)}-\al^{(j)})^2=D(h).
\end{eqnarray*}
Hence $D(h)\in \fd_{L/K}$.
\end{proof}

Put $u(n):=\lcm(1,2,\dots ,n)$.
For the possible prime factors of the discriminant $\fd_{L/K}$ we have:

\begin{lemma}\label{L_gen_disc_primedivisors}
Let $[L:K]=n$. Then
for every prime ideal $\fp \in \cP(K)$ with $\ord_{\fp}(\fd_{L/K})>0$ we have
$$
\ord_{\fp}(\fd_{L/K})\leq n\cdot (1+\ord_{\fp}(u(n))).
$$
\end{lemma}

\begin{proof}
Let $\fD_{L/K}$ denote the different of $L/K$.
According to J. Neukirch \cite[p. 210, Theorem 2.6]{Neukirch1},
we have for every prime ideal $\fP$ of $L$ lying above $\fp$
\begin{eqnarray*}
\ord_{\fP}(\fD_{L/K})&\leq& e(\fP | \fp)-1+\ord_{\fP}(e(\fP | \fp))
\\
&\leq& e(\fP | \fp)-1+e(\fP | \fp) \ord_{\fp}(e(\fP | \fp)),
\end{eqnarray*}
where $e(\fP |\fp )$, $f(\fP |\fp)$ denote the ramification index and residue class degree
of $\fP$ over $\fp$.
Using $\fd_{L/K}=N_{L/K}\fD_{L/K}$, $N_{L/K}\fP=\fp^{f(\fP |\fp )}$,\\
$\sum_{\fP|\fp } e(\fP |\fp )f(\fP |\fp )=[L:K]\leq n$, we infer
$$
\begin{aligned}
\ord_{\fp}(\fd_{L/K})&=\ord_{\fp}(N_{L/K}\fD_{L/K})=\sum_{\fP | \fp} f(\fP | \fp)\ord_{\fP}(\fD_{L/K}) \\
&\leq \sum_{\fP | \fp} f(\fP | \fp)e(\fP | \fp)(1+\ord_{\fp}(e(\fP | \fp))
\\
&\leq n(1+\ord_{\fp}(u(n))).
\end{aligned}
$$
\end{proof}

\begin{lemma}\label{L_gen_disc_extensions}
{\rm (i)} Let $M\supset L \supset K$ be a tower of finite extensions. Then we have
$$
\fd_{M/K} = N_{L/K}(\fd_{M/L}) \fd_{L/K}^{[M:L]}  .
$$
{\rm (ii)} Let $L_1, L_2$ be finite extensions of $K$.
Then for their compositum $L_1 \cdot L_2$ we have
$$
\fd_{L_1L_2/K} \supseteq \fd_{L_1/K}^{[L_1L_2:L_1]} \fd_{L_2/K}^{[L_1L_2:L_2]}.
$$
\end{lemma}

\begin{proof}
For (i) see Neukirch \cite[p. 213, Korollar 2.10]{Neukirch1}.
For (ii) apply Stark \cite[Lemma 6]{Stark1} and take norms.
\end{proof}

\begin{lemma}\label{L_gen_disc_primes_excluded}
Let $m\in \Z_{\geq 0}$, $\ga \in K^*$ and $L:=K(\sqrt[m]{\ga})$. Further, let $\fp \in \cP(K)$ be a prime ideal with
$$
\ord_{\fp}(m)=0,\ \ \ \  \ord_{\fp}(\ga) \equiv 0 \pmod m.
$$
Then $L/K$ is unramified at $\fp$, i.e.
$$
\ord_{\fp}(\fd_{L/K}) = 0.
$$
\end{lemma}

\begin{proof}
Choose $\tau \in K^*$ such that $\ord_{\fp}(\tau)=1$. Then $\ga =\tau^{mt}\ve$ with $t \in \Z$ and $\ord_{\fp}(\ve)=0$.
We clearly have $L=K(\sqrt[m]{\ve})$, hence
$$
\fd_{L/K} \supseteq \frac{[D(X^m-\ve)]}{[1, \ve]^{2m-2}} = \frac{[m^m\ve^{m-1}]}{[1, \ve]^{2m-2}}.
$$
This implies $\ord_{\fp} (\fd_{L/K})=0$.
\end{proof}

\subsection{$S$-integers}

Let $K$ be an algebraic number field and denote by $M_K$ its set of places.
We keep using throughout the absolute values defined by \eqref{abs_val}.
Recall that these absolute values satisfy the product formula
\[
\prod_{v\in M_K} |\al |_v=1\ \ \mbox{for } \al\in K^* .
\]
If $L$ is a finite extension of $K$, and $v,w$ places of $K,L$, respectively,
we say that $w$ lies above $v$, notation $w|v$,
if the restriction of $|\cdot |_w$ to $K$
is a power of $|\cdot |_v$, and in that case we have
\[
|\al |_w=|\al |_v^{[L_w:K_v]}\ \ \mbox{for } \al\in K,
\]
where $K_v,L_w$ denote the completions of $K$ at $v$, $L$ at $w$, respectively.
In case that $v=\fp$, $w=\fP$ are  prime ideals of $O_K,O_L$, respectively,
we have $w|v$ if and only if $\fp\subset\fP$.

Let $S$ be a finite set of places of $K$ containing all infinite places.
The non-zero fractional ideals of the ring of $S$-integers $O_S$
(i.e., finitely generated $O_S$-submodules of $K$) form a group under multiplication, and there is
an isomorphism from the multiplicative group
of non-zero fractional ideals of $O_S$ to the group of fractional ideals of $O_K$ composed of prime ideals outside $S$
given by $\fa\mapsto \fa^*$, where $\fa =\fa^*O_S$. We define the $S$-norm
of a fractional ideal of $O_S$ by
\[
N_S(\fa ):=N_K\fa^*=\mbox{absolute norm of $\fa^*$.}
\]
Given $\alpha_1\kdots\alpha_r\in K$ we denote by
$[\alpha_1\kdots\alpha_r]_S$ the fractional ideal of $O_S$ generated
by $\alpha_1\kdots\alpha_r$.
We have
\begin{equation}\label{3.1}
N_S([\alpha_1\kdots\alpha_r]_S)=\prod_{v\in M_K\setminus S}\max (|\alpha_1|_v\kdots |\alpha_r|_v)^{-1}.
\end{equation}
Further, for $\alpha\in K$ we define $N_S(\alpha ):=N_S([\alpha ]_S)$.
By the product formula,
\begin{equation}\label{3.2}
N_S(\alpha )=\prod_{v\in S} |\alpha |_v\ \ \mbox{for } \alpha\in K.
\end{equation}

Let $L$ be a finite extension of $K$, and $T$ the set of places of $L$ lying above
the places in $S$. Then the ring of $T$-integers $O_T$ is the integral closure in $L$
of $O_S$. Every fractional ideal $\fA$ of $O_T$ can be expressed uniquely as
$\fA =\fA^*O_T$ where $\fA^*$ is a fractional ideal of $O_L$ composed of prime ideals
outside $T$. We put
\[
N_T\fA := N_L\fA^*,\ \ \ N_{T/S}\fA := (N_{L/K}\fA^*)O_S.
\]
Then
\begin{equation}\label{norm-relation}
\left\{\begin{array}{l}
N_T\fA=N_S(N_{T/S}\fA),
\\
N_T(\fa O_T)=N_S\fa^{[L:K]}\ \mbox{for a fractional ideal $\fa$ of $O_S$.}
\end{array}\right.
\end{equation}

Let $\fp_1\kdots \fp_t$ be the prime ideals in $S$ and put $Q_S:=\prod_{i=1}^t N_K\fp_i$.
Let $\fP_1\kdots \fP_{t'}$ be the prime ideals in $T$ and put
$Q_T:=\prod_{i=1}^{t'} N_K\fP_i$. Then
for every prime ideal $\fp$ of $O_K$ we have
\[
\prod_{\fP \mid \fp}N_L\fP =\prod_{\fP \mid \fp}(N_K\fp )^{f_{\fP\mid\fp}}
\leq \prod_{\fP \mid \fp}(N_K\fp )^{e_{\fP\mid\fp}\cdot f_{\fP\mid\fp}}
\leq (N_K\fp )^{[L:K]},
\]
where the product is over all prime ideals $\fP$ of $O_L$ dividing $\fp$
and where $e(\fP |\fp)$, $f(\fP |\fp )$ denote the ramification index and residue class degree of $\fP$ over $\fp$. Hence
\begin{equation}\label{bound-3}
Q_T\leq Q_S^{[L:K]}.
\end{equation}

\subsection{Class number and regulator}

Let again $K$ be a number field.

\begin{lemma}\label{L_hR}
For the regulator $R_K$ and class number $h_K$ of $K$ we have the following estimates:
\begin{align}
&R_K \geq 0.2, \label{regulator_lower_bound}\\
&h_KR_K \leq |D_K|^{\frac{1}{2}}(\log^* |D_K|)^{d-1} \label{bound_hkRk}.
\end{align}
\end{lemma}

\begin{proof}
Statement \eqref{regulator_lower_bound} is a result of Friedman \cite{Friedman1}. Inequality \eqref{bound_hkRk} follows from Louboutin \cite{Louboutin1},
see also (59) in Gy\H ory and Yu \cite{GyYu1}.
\end{proof}

Let $S$ be a finite set of places of $K$ consisting of the infinite places
and of the prime ideals $\fp_1\kdots\fp_t$.
Then the $S$-regulator $R_S$ is given by
\begin{equation}\label{3.3}
R_S=h_SR_K\prod_{i=1}^t\log N_K\fp_i,
\end{equation}
where $h_S$ is the order of the group generated by the ideal classes of
$\fp_1\kdots\fp_t$ and where $h_S$ and the product are $1$ if $S$ consists only
of the infinite places.
Together with
Lemma \ref{L_hR} this implies
\begin{equation}\label{3.4}
\mbox{$\frac{1}{5}\ln 2$}\leq R_S\leq |D_K|^{\frac{1}{2}}(\log^* |D_K|)^{d-1}\cdot
(\log P_S)^t,
\end{equation}
where the last factor has to be interpreted as $1$ if $t=0$.

\subsection{Heights}

We define the absolute logarithmic height of $\al \in \OQq$ by
$$
h(\al)=\frac{1}{[K:\Qq ]}\sum_{v \in M_K} \max (0,\log |\al|_v),
$$
where $K$ is any number field with $K\ni\al$.
More generally, we define the logarithmic height of a polynomial
$f(X)=a_0x^n+\dots +a_n\in \OQq [X]$ by
$$
h(f):=\frac{1}{[K:\Qq ]}\sum_{v \in M_K} \log \max (1,|a_0|_v, \dots ,|a_n|_v)
$$
where $K$ is any number field with $f\in K[X]$. These heights do not depend
on the choice of $K$.

We will frequently use the
inequalities
\[
h(\alpha_1\cdots\alpha_n)\leq \sum_{i=1}^n h(\alpha_i),\ \ \
h(\alpha_1+\cdots +\alpha_n)\leq \sum_{i=1}^n h(\alpha_i)+\log n
\]
for $\alpha_1\kdots\alpha_n\in\OQq$ and the equality
\[
h(\alpha^m)=|m|h(\alpha )\ \ \mbox{for } \alpha\in\OQq^*,\, m\in\Zz .
\]
(see Waldschmidt \cite[Chapter 3]{Waldschmidt2000}).
Further we frequently use the trivial fact that if $\alpha$ belongs
to a number field $K$ and $S$ is a finite set of places of $K$ containing
the infinite places, then
\[
h(\alpha)\geq \frac{1}{[K:\Qq ]}\log N_S(\alpha ).
\]

We have collected some further facts.

\begin{lemma}\label{L_bound_roots_by_pol}
Let $\alpha_1\kdots\alpha_n\in\OQq$ and $f=(X-\alpha_1)\cdots (X-\alpha_n)$.
Then
\[
|h(f)-\sum_{i=1}^n h(\alpha_i)|\leq n\log 2.
\]
\end{lemma}

\begin{proof}
See Bombieri and Gubler \cite[p.28, Thm.1.6.13]{BombGub}.
\end{proof}

\begin{lemma}\label{disc-height}
Let $K$ be a number field and $f=a_0X^n+a_1X^{n-1}+\cdots +a_n\in K[X]$
a polynomial of degree $n$ with discriminant $D(f)\not= 0$.
Then
\begin{eqnarray*}
{\rm (i)}&&|D(f)|_v\leq n^{(2n-1)s(v)}\max (|a_0|_v\kdots |a_n|_v)^{2n-2}\ \ \mbox{for } v\in M_K,
\\
{\rm (ii)}&&h(D(f))\leq (2n-1)\log n +(2n-2)h(f),
\end{eqnarray*}
where $s(v)=1$ if $v$ is real, $s(v)=2$ if $v$ is complex, $s(v)=0$ if $v$ is finite.
\end{lemma}

\begin{proof}
Inequality (ii) is an immediate consequence of (i).
For finite $v$, inequality (i) follows from the ultrametric inequality, noting that
$D(f)$ is a homogeneous polynomial of degree $2n-2$ in the coefficients of $f$
with integer coefficients. For infinite $v$, inequality (i) follows from a
a result of Lewis and Mahler \cite[p. 335]{LewisMahler}).
\end{proof}

\begin{lemma}\label{L_bound_hight_by_norm}
Let $K$ be an algebraic number field and $S$ a finite set of places of $K$,
which consists of the infinite places and of the prime ideals $\fp_1\kdots\fp_t$.
Then
for every $\al\in O_S \setminus \{ 0 \}$ and $m\in \NN$ there exists an $S$-unit $\eta \in O_S^*$ with
$$
h(\al\eta^m)\leq \frac{1}{d}\log N_S(\al)+m\cdot \left( cR_K+\frac{h_K}{d}
\log Q_S \right),
$$
where $c:=39d^{d+2}$ and $Q_S:=\prod_{i=1}^t N_K\fp_i$.
\end{lemma}

\begin{proof}
This is a slightly weaker version of Lemma 3 of Gy\H ory and Yu \cite{GyYu1}.
The result was essentially proved (with a larger constant)
in \cite{BugGy2} and \cite{Gy21}.
\end{proof}

\begin{lemma}\label{L_lower bound_for hight}
Let $\al$ be a non-zero algebraic number of degree $d$ which is not a root of unity. Then
$$
h(\al) \geq m(d):=
\begin{cases}
\log 2 &\text{if} \quad d = 1, \\
2/d(\log 3d)^3 &\text{if} \quad d \geq 2.
\end{cases}
$$
\end{lemma}

\begin{proof}
See Voutier \cite{Voutier1}.
\end{proof}

\subsection{Baker's method}

Let $K$ be an algebraic number field, and denote by $M_K$ the set of places of $K$.
Let $\al_1, \dots ,\al_n$ be  $n\geq 2$ non-zero elements of $K$, and $b_1,\dots ,b_n$
are rational integers, not all zero. Put
\begin{eqnarray*}
&&\Lambda := \al_1^{b_1} \dots \al_n^{b_n}-1,
\\
&&\Theta :=\prod_{i=1}^n \max \big( h(\alpha_i), m(d)\big),
\\
&&B:=\max (3,|b_1|,\kdots |b_n|),
\end{eqnarray*}
where $m(d)$ is the lower bound from Lemma \ref{L_lower bound_for hight}
(i.e., the maximum is $h(\alpha_i)$ unless $\alpha_i$ is a root of unity).
For a place $v \in M_K$, we write
$$
N(v) =
\begin{cases}
2 & \text{if $v$ is infinite} \\
N_K\fp & \text{if $v=\fp$ is finite}.
\end{cases}
$$

\begin{proposition}\label{P_Baker}
Suppose that $\Lambda \ne 0$. Then for $v\in M_K$ we have
\begin{equation}\label{lowerbound-Lambda}
\log |\Lambda|_v > -\, c_1(n, d)\frac{N(v)}{\log N(v)} \Theta\log B ,
\end{equation}
where $c_1(n,d) =  12(16ed)^{3n+2} (\log^* d)^2.$
\end{proposition}

\begin{proof}
First assume that $v$ is infinite. Without loss of generality, we assume
that $K\subset\Cc$ and $|\cdot |_v=|\cdot |^{s(v)}$ where $s(v)=1$ if
$K\subset\Rr$ and $s(v)=2$ otherwise.
Denote by $\log$ the principal natural logarithm on $\Cc$
(with $|{\rm Im}\, \log z|\leq\pi$ for $z\in\Cc^*$.
Let $b_0$ be the rational integer such that $|{\rm Im}\, \Xi |\leq\pi$, where
\[
\Xi :=b_1\log \alpha_1 +\cdots +b_n\log\alpha_n +2b_0\log (-1),\ \ \
\log (-1)=\pi i.
\]
Thus,
\[
B':=\max (|2b_0|,|b_1|\kdots |b_n|)\leq 1+nB.
\]
A result of Matveev \cite[Corollary 2.3]{Matveev1} implies that
\[
\log |\Xi |\geq -\, s(v)^{-1}\big(\mbox{$\frac{1}{2}$}e(n+1)\big)^{s(v)}
(n+1)^{3/2}30^{n+4}d^2(\log ed)\Omega\log (eB'),
\]
where
\[
\Omega :=\pi\prod_{i=1}^n\max (h(\alpha_i),\pi ).
\]
Assuming, as we may, that $|\Lambda |\leq \half$,
we get $|\Xi |=|\log (1+\Lambda )|\leq 2|\Lambda |\leq 1$. Further,
$\Omega\leq \pi^{n+1}m(d)^{-n}\Theta$. By combining this with Matveev's lower bound
we obtain a  lower bound for $|\Lambda |_v$
which is better than \eqref{lowerbound-Lambda}.

Now assume that $v$ is finite, say $v=\fp$, where $\fp$ is a prime ideal of $O_K$.
By a result of K. Yu \cite{Yu1} (consequence of Main Theorem on p. 190) we have
\[
\ord_{\fp}(\Lambda )\leq (16ed)^{2n+2}n^{3/2}\log (2nd)\log (2d)e_{\fp}^n\cdot
\frac{N_K\fp}{(\log N_K\fp )^2}\cdot
\Theta\log B ,
\]
where $e_{\fp}$ is the ramification index of $\fp$.
Using that $\log |\Lambda |_{\fp}=-\ord_{\fp}(\Lambda )\log N_K\fp$ and $e_{\fp}\leq d$,
we obtain a lower bound for $\log |\Lambda |_{\fp}$ which is better
than \eqref{lowerbound-Lambda}.
\end{proof}

\subsection{Thue equations and Pell equations}\label{thue-pell}

Let $K$ be an algebraic number field of degree $d$, discriminant $D_K$, regulator $R_K$ and class number $h_K$, and denote by $O_K$ its ring of integers. Let $S$ be a finite set of places of $K$ containing all infinite places. Denote by $s$ the cardinality of $S$ and by $O_S$ the ring of $S$ integers in $K$. Further denote by $R_S$ the $S$-regulator,
let $\fp_1\kdots\fp_t$ be the prime ideals in $S$, and put
$$
P_S:=\max\{ N_K\fp_1\kdots N_K\fp_t\},\ \ \ Q_S:=N_K(\fp_1\cdots \fp_t),
$$
with the convention that $P_S=Q_S=1$ if $S$ contains no finite places.

We state effective results on Thue equations and on systems of Pell equations
which are easy consequences of a general effective result on decomposable form
equations by Gy\H{o}ry and Yu \cite{GyYu1}. In both results we use the constant
\[
c_1(s,d):= s^{2s+4}2^{7s+60}d^{2s+d+2}.
\]

\begin{proposition}\label{P_Thue}
Let $\beta\in K^*$ and let $F(X,Y)=\sum_{i=0}^n a_iX^{n-i}Y^i\in K[X,Y]$ be a binary form of degree $n\geq 3$
with non-zero discriminant which splits into linear factors over $K$.
Suppose that
\[
\max_{0\leq i\leq n} h(a_i)\leq A,\ \ \ h(\beta )\leq B.
\]
Then for the solutions of
\begin{equation}\label{Thue}
F(x,y)=\beta\ \ \ \mbox{in } x,y\in O_S
\end{equation}
we have
\begin{eqnarray}
\label{Thue-bound}
&&\max (h(x),h(y))
\\
\nonumber
&&
\leq c_1(s,d)n^6P_SR_S\left( 1+\frac{\log^* R_S}{\log^* P_S}\right)\cdot \Big( R_K+\frac{h_K}{d}\log Q_S+ndA+B\Big).
\end{eqnarray}
\end{proposition}

\begin{proof}
Gy\H{o}ry and Yu \cite[p. 16, Corollary 3]{GyYu1}
proved this with instead of our $c_1(s,d)$
a smaller bound $5d^2n^5\cdot 50(n-1)c_1c_3$,
where $c_1,c_3$
are given respectively in \cite[Theorem 1]{GyYu1}, and in
\cite[bottom of page 11]{GyYu1}.
\end{proof}

\begin{proposition}\label{P_Pell}
Let $\gamma_1,\gamma_2,\gamma_3,\beta_{12},\beta_{13}$ be non-zero elements of $K$
such that
\begin{eqnarray*}
&\beta_{12}\not=\beta_{13},\ \ \
\sqrt{\gamma_1/\gamma_2},\, \sqrt{\gamma_1/\gamma_3}\in K,&
\\
&h(\gamma_i)\leq A\ \mbox{for } i=1,2,3,\ \ \ h(\beta_{12}),h(\beta_{13})\leq B.
\end{eqnarray*}
Then for the solutions of the system
\begin{equation}\label{Pell}
\gamma_1x_1^2-\gamma_2x_2^2=\beta_{12},\ \ \gamma_1x_1^2-\gamma_3x_3^2=\beta_{13}\ \ \ \
\mbox{in } x_1,x_2,x_3\in O_S
\end{equation}
we have
\begin{eqnarray}
\label{Pell-bound}
&&\max (h(x_1),h(x_2),h(x_3))
\\
\nonumber
&&
\leq c_1(s,d)P_SR_S\left( 1+\frac{\log^* R_S}{\log^* P_S}\right)
\cdot \Big( R_K+\frac{h_K}{d}\log Q_S+dA+B\Big).
\end{eqnarray}
\end{proposition}

\begin{proof}
Put $\beta_{23}:=\beta_{13}-\beta_{12}$, $\beta :=\beta_{12}\beta_{13}\beta_{23}$
and define
\[
F:=(\gamma_1X_1^2-\gamma_2X_2^2)(\gamma_1X_1^2-\gamma_3X_3^2)(\gamma_2X_2^2-\gamma_3X_3^2).
\]
Thus, every solution of \eqref{Pell} satisfies also
\begin{equation}\label{Pell-decomp}
F(x_1,x_2,x_3)=\beta\ \ \ \mbox{in } x_1,x_2,x_3\in O_S.
\end{equation}
By assumption, $\beta\not= 0$.
Further, $F$ is a decomposable form of degree $6$
with splitting field $K$,
i.e., $F=l_1\cdots l_6$ where $l_1\kdots l_6$ are linear forms
with coefficients in $K$. We make a graph on $\{ l_1\kdots l_6\}$ by connecting two
linear forms $l_i,l_j$ if there is a third linear form $l_k$
such that $l_k=\lambda l_i+\mu l_j$ for certain non-zero $\lambda ,\mu\in K$.
Then this graph is connected.
Further, $\rank\{ l_1\kdots l_6\}=3$.
Hence $F$ satisfies all the conditions of
Theorem 3 of Gy\H{o}ry and Yu \cite{GyYu1}. According to this Theorem,
the solutions $x_1,x_2,x_3$ of \eqref{Pell-decomp}, and so also the solutions
of \eqref{Pell},  satisfy \eqref{Pell-bound}
but with instead of $c_1(s,d)$
the smaller number
$375c_1c_3$,
where $c_1,c_3$
are given respectively in \cite[Theorem 1]{GyYu1}, and on
\cite[bottom of page 11]{GyYu1}.
\end{proof}

\section{Proof of the results in the case of fixed exponent}

Let $K$ be an algebraic number field, put $d:=[K:\Qq ]$, and let
$D_K$ denote the discriminant of $K$.
Further, let $S$ be a finite set of places of $K$ containing all infinite places.

\begin{lemma}\label{L_disc_I}
Let $f(X)\in K[X]$ be a polynomial of degree $n$ and discriminant $D(f)\ne 0$. Suppose that $f$ factorizes over an extension of $K$ as
$a_0(X-\al_1)\dots(X-\al_n)$ and let $L:=K(\al_1, \dots, \al_k)$. Then for the
discriminant of $L$ we have
$$
|D_L|\leq \left( n\cdot e^{h(f)} \right)^{2kn^kd} \cdot |D_K|^{n^k}.
$$
For the case $k=1$ we have the sharper estimate
$$
|D_L|\leq n^{(2n-1)d} \cdot e^{(2n-2)d\cdot h(f)} \cdot |D_K|^{[L:K]}.
$$
\end{lemma}

\begin{proof}
By Lemma \ref{L_gen_disc_extensions} (i), we have
\begin{equation}\label{disc_0}
|D_L|=N_K\fd_{L/K} \cdot |D_K|^{[L:K]}\leq N_K\fd_{L/K} \cdot |D_K|^{n^k}.
\end{equation}
Applying Lemma \ref{L_gen_disc_extensions} (ii) to $L=K(\alpha_1)\cdots K(\alpha_k)$
yields
\begin{equation}\label{disc_1}
\fd_{L/K} \supseteq \prod_{i=1}^k \left( \fd_{K(\al_i)/K}\right)^{[L:K(\al_i)]}.
\end{equation}
Further, since $\al_i$ is a root of $f$ we have by Lemma \ref{L_gen_disc_I},
$$
\fd_{K(\al_i)/K} \supseteq \frac{[D(f)]}{[f]^{2n-2}},
$$
and so
\begin{equation}\label{disc_2}
N_K\fd_{K(\al_i)/K}\leq  N_K\left( \frac{[D(f)]}{[f]^{2n-2}} \right).
\end{equation}
By Lemma \ref{disc-height} we have
\begin{eqnarray*}
|N_K(D(f))|&=&\prod_{v \in M_K^{\infty}} |D(f)|_v \leq \prod_{v \in M_K^{\infty}} \left(n^{2n-1}\right)^{s(v)}|f|_v^{2n-2}
\\
&\leq& n^{(2n-1)d}\prod_{v \in M_K^{\infty}} |f|_v^{2n-2}
\end{eqnarray*}
where $|f|_v$ is the maximum of the $v$-adic absolute values of the coefficients of $f$;
moreover,
$$
N_K([f]^{-2n+2})=\prod_{v\in M_K \setminus M_K^{\infty}} |f|_v^{2n-2}.
$$
Thus, we obtain
\begin{equation}\label{disc_3}
 N_K\left( \frac{[D(f)]}{[f]^{2n-2}} \right)  \leq \left( n^{2n-1} \cdot e^{(2n-2)h(f)} \right)^{d}.
\end{equation}
Together with \eqref{disc_0}, \eqref{disc_2} this implies
the sharper upper bound for $|D_L|$ in the case $k=1$.
For arbitrary $k$, combining \eqref{disc_1}, \eqref{disc_2}, \eqref{disc_3}
and the estimate $[L:K(\al_i)] \leq (n-1)(n-2)\cdots (n-k+1)$ gives
$$
\begin{aligned}
N_K\fd_{L/K} &\leq  \left( n^{2n-1} \cdot e^{(2n-2)h(f)} \right)^{k(n-1)(n-2)\cdots (n-k+1)d} \\
&\leq n^{k(2n-1)n^{k-1}d} \cdot e^{k(2n-2)n^{k-1}d\cdot h(f)} \leq \left( n\cdot e^{h(f)} \right)^{2kn^kd}.
\end{aligned}
$$
This in turn, together with \eqref{disc_0} proves Lemma \ref{L_disc_I}.
\end{proof}

Let
\[
f=a_0X^n+a_1X^{n-1}+\dots+a_n\in O_S[X]
\]
be a polynomial of degree $n\geq 2$ with discriminant $D(f)\not= 0$.
Let $b$ be a
non-zero element of $O_S$, $m$ an integer $\geq 2$ and consider the equation
\begin{equation}\label{gen_ell}
f(x)=by^m \quad \quad \text{in $x,y\in O_S$. }
\end{equation}
Put
\begin{equation}\label{h-hat}
\widehat{h}:=\frac{1}{d}\sum_{v \in M_K} \log \max (1,|b|_v,|a_0|_v\kdots |a_n|_v).
\end{equation}
Let $G$ be the splitting field of $f$ over $K$. Then
\[
f=a_0(X-\alpha_1)\cdots (X-\alpha_n)\ \
\mbox{with } \alpha_1\kdots\alpha_n\in G.
\]
For $i=1\kdots n$, let $L_i=K(\alpha_i)$ and denote by $T_i$ the set
of places of $L_i$ lying above the places of $S$.
We denote by $[\beta_1\kdots\beta_r]_{T_i}$ the
fractional of $O_{T_i}$ generated by $\beta_1\kdots \beta_r$.
Then we have the following Lemma:

\begin{lemma}\label{L_gen_ell}
Let $x,y \in O_S$ be a solution of equation \eqref{gen_ell} with $y\not= 0$.
Then for $i=1\kdots n$ we have the following:
\\
(i) There are ideals $\fC_i$, $\fA_i$ of $O_{T_i}$ such that
\begin{equation}\label{split_to_ideal_powers}
[a_0(x-\alpha_i )]_{T_i}=\fC_i\fA_i^m,\ \ \ \fC_i\supseteq [a_0bD(f)]_{T_i}^{m-1}.
\end{equation}
(ii) There are $\ga_i,\, \xi_i$ with
\begin{equation}\label{split_to_powers}
\left\{\begin{array}{l}
x-\al_i =\ga_i\xi_i^m,\ \ \ga_i\in L_i^*,\, \xi\in O_{T_i},
\\[0.2cm]
h(\ga_i) \leq m(n^3d)^{nd}e^{2nd\widehat{h}}|D_K|^n\cdot
\Big( 80(dn)^{dn+2}+\frac{1}{d}\log Q_S\Big).
\end{array}\right.
\end{equation}
\end{lemma}

\begin{proof}
It suffices to prove the Lemma for $i=1$.
We suppress the index $1$ and write $\alpha ,T,L,\gamma,\xi$
for $\alpha_1,T_1,L_1,\gamma_1,\xi_1$.
Let  $g:=(X-\al_2)\dots (X-\al_n)$. By $[\cdot ]$ we denote fractional ideals
in $G$ with respect to the integral closure of $O_T$ in $G$. Clearly,
\[
\frac{[x-\alpha]}{[1,\alpha]}+\frac{[x-\alpha_i]}{[1,\alpha_i]}\supseteq
\frac{[\alpha -\alpha_i]}{[1,\alpha ][1,\alpha_i]}
\]
for $i=2\kdots n$. This implies
\[
\frac{[x-\alpha]}{[1,\alpha]}+\prod_{i=2}^n\frac{[x-\alpha_i]}{[1,\alpha_i]}\supseteq
\prod_{i=2}^n \frac{[\alpha -\alpha_i]}{[1,\alpha ][1,\alpha_i]}
\]
Noting that by Gauss' Lemma we have $[f]=[a_0]\prod_{i=1}^n [1,\alpha_i]$,
we see that the right-hand side contains
\[
\prod_{j=1}^n\prod_{i\not= j}\frac{[\alpha_j-\alpha_i]}{[1,\alpha_j][1,\alpha_i]}
=\frac{[D(f)]}{[f]^{2n-2}}.
\]
Using also $[g]=\prod_{i=2}^n [1,\alpha_i]$ we obtain
\begin{equation}\label{g-sum}
\frac{[x-\alpha]}{[1,\alpha]}+\frac{[g(x)]}{[g]}\supseteq \frac{[D(f)]}{[f]^{2n-2}}.
\end{equation}
Writing equation \eqref{gen_ell} as equation of ideals, we get
\begin{equation}\label{g-product}
[b][f]^{-1}[y]^m=\frac{[x-\alpha]}{[1,\alpha]}\cdot \frac{[g(x)]}{[g]}.
\end{equation}
Note that the ideals occurring in \eqref{g-sum}, \eqref{g-product}
are all defined over $L$, so we may view them as ideals of $O_T$.
Henceforth, we use $[\cdot ]$ to denote ideals of $O_T$.

Now let $\fP$ be a prime ideal of $O_T$ not dividing $a_0bD(f)$.
Note that $D(f)\in [f]^{2n-2}$, hence $\fP$ does not divide $[f]$ either.
By \eqref{g-sum}, the prime ideal
$\fP$ divides at most one of the ideals $\frac{[x-\al_1]}{[1,\al_1]}$
and $\frac{[g(x)]}{[g]}$, and we get
$$
\ord_{\fP} \frac{[x-\al ]}{[1,\al ]} \equiv 0 \pmod m.
$$
But $[a_0][1,\alpha ]$ is not divisible by $\fP$ since it contains $a_0$.
Hence
$$
\ord_{\fP}(a_0(x-\al )) \equiv 0 \pmod m.
$$
Applying division with remainder to the exponents of the prime ideals
dividing $a_0bD(f)$ in the factorization of $a_0(x-\alpha )$,
we obtain
that there are ideals $\fC$, $\fA$ of $O_T$, with $\fC$ dividing
$(ba_0D(f))^{m-1}$ such that $[a_0(x-\al )]=\fC\fA^m$.
This proves (i).

We prove (ii).
The ideal $\fA$ of $O_T$ may be written as $\fA=\fA^* O_T$ with an ideal $\fA^*$
of $O_L$ composed of prime ideals outside $T$, and further, we may choose
non-zero $\xi_1\in\fA^*$ with $|N_{L/\Qq}(\xi_1 )|\leq |D_L|^{1/2}N_L\fA^*$
(see Lang \cite[pp. 119/120]{Lang70}.
This implies $N_T(\xi_1 )\leq |D_L|^{1/2}N_T\fA$, i.e.,
$[\xi_1 ]=\fB\fA$ where $\fB$ is an ideal of $O_T$ with $N_T\fB\leq |D_L|^{1/2}$.
Similarly, there exists $\gamma_1\in L$ with $[\gamma_1]=\fD\fC$, where
$\fD$ is an ideal of $O_T$ with $N_T\fD\leq |D_L|^{1/2}$.
As a consequence, we have
\[
a_0(x-\al )=\frac{\gamma_1}{\gamma_2}\xi_1^m,
\]
where $\gamma_1,\gamma_2\in O_T$, and
\[
[\gamma_2]=\fD\fB^m.
\]
Using (i) and the choice of $\fB$, $\fD$, we get
\begin{equation}\label{bound_norm}
N_T(\gamma_1)\leq |D_L|^{1/2}N_T(a_0bD(f))^{m-1},\ \ \ N_T(\gamma_2)\leq |D_L|^{(m+1)/2}.
\end{equation}
According to
Lemma \ref{L_bound_hight_by_norm} we can find $T$-units $\eta_1,\eta_2\in O_T^*$ such that
\[
h(\gamma_i\eta_i^m) \leq d_L^{-1}\log N_T(\gamma_i)+m\cdot \left( cR_L+\frac{h_L}{d_L} \log Q_T \right)\ \,\mbox{for } i=1,2
\]
where $d_L=[L:\Qq ]$, $c:=39d_L^{d_L+2}$ and $Q_T:=\prod\limits_{\ketsor{\fP\in T}{\fP \text{ finite}}} N_L\fP$.
Putting
\[
\gamma :=a_0^{-1}\gamma_1\gamma_2^{-1}(\eta_1\eta_2^{-1})^m,\ \ \
\xi = \eta_2\eta_1^{-1}\xi_1,
\]
and invoking \eqref{bound_norm}
we obtain $x-\al =\gamma\xi^m$, with $\xi\in O_T$, $\gamma\in L^*$ and
\begin{eqnarray}\label{bound_heightgamma}
h(\gamma )&\leq& h(a_0)+d_L^{-1}\Big(\frac{m+1}{2}\log |D_L|+m\log N_T(abD(f))\Big)+
\\
\nonumber
&&\qquad
+2m\cdot \Big( cR_L+\frac{h_L}{d_L} \log Q_T \Big).
\end{eqnarray}

It remains to estimate from above the right-hand side of \eqref{bound_heightgamma}.
First, we have by \eqref{norm-relation} and Lemma \ref{disc-height},
\begin{eqnarray}\label{a0bD(f)}
d_L^{-1}\log N_T(a_0bD(f))&=&d^{-1}\log N_S(a_0bD(f))\leq h(a_0bD(f))
\\
\nonumber
&\leq& (2n-1)\log n +2n\widehat{h}.
\end{eqnarray}
Together with Lemma \ref{L_disc_I} this implies
\begin{eqnarray}\label{bound-1}
&&h(a_0)+d_L^{-1}\Big(\frac{m+1}{2}\log |D_L|+m\log N_T(abD(f))\Big)
\\
\nonumber
&&\qquad \leq m(4n\log n +4n\widehat{h}+\log |D_K|).
\end{eqnarray}
Next, by Lemma \ref{L_hR}, Lemma \ref{L_disc_I} and $d_L\leq nd$ we have
\begin{eqnarray}\label{bound-2}
\max (h_L,R_L)
&\leq& 5|D_L|^{1/2}(\log^* |D_L|)^{nd-1}\leq (nd)^{nd}|D_L|
\\
\nonumber
&\leq& (n^3d)^{nd}e^{(2n-2)d\widehat{h}}|D_K|^n.
\end{eqnarray}
By inserting the bounds \eqref{bound-1}, \eqref{bound-2},
together with \eqref{bound-3} and the estimate
$c\leq 39(nd)^{nd+2}$ into \eqref{bound_heightgamma},
one easily obtains the upper bound for $h(\gamma )$ given by (ii).
\end{proof}

Let $f$, $b$, $m$ be as above, and let $x,y\in O_S$ be a solution of
\eqref{gen_ell} with $y\not= 0$.
Let $\gamma_1\kdots\gamma_n$, $\xi_1\kdots\xi_n$ be as in Lemma \ref{L_gen_ell}.

\begin{lemma}\label{L_discresult}
(i) Let $m\geq 3$ and
$M=K(\alpha_1,\alpha_2,\root\m\of{\gamma_1/\gamma_2} ,\rho )$,
where $\rho$ is a primitive $m$-th root of unity. Then
\begin{equation}\label{discestimate-1}
|D_M|\leq 10^{m^3n^2d}n^{4m^2n^3d}|D_K|^{m^2n^2}Q_S^{m^2n^2}e^{4m^2n^3d\widehat{h}}.
\end{equation}
(ii) Let $m=2$ and
$M=K(\alpha_1,\alpha_2,\alpha_3,\sqrt{\gamma_1/\gamma_2},\sqrt{\gamma_1/\gamma_3})$.
Then
\begin{equation}\label{discestimate-2}
|D_M|\leq n^{40n^4d}Q_S^{8n^3}|D_K|^{4n^3}e^{25n^4d\widehat{h}}.
\end{equation}
\end{lemma}

\begin{proof}
We start with (i). Define the fields $L=K(\alpha_1,\alpha_2)$,
$M_1=L(\root\m\of{\gamma_1/\gamma_2})$, $M_2=L(\rho )$. Then $M=M_1M_2$.
By Lemma \ref{L_gen_disc_extensions} (i) we have
\begin{equation}\label{discrelation-1}
|D_M|=N_L\fd_{M/L}|D_L|^{[M:L]}.
\end{equation}
By Lemma \ref{L_gen_disc_I}, we have
$\fd_{M_2/L}\supseteq [m]^m$, where $[m]=mO_L$.
Together with
Lemma \ref{L_gen_disc_extensions} (ii), this implies
\[
\fd_{M/L}\supseteq\fd_{M_1/L}^{[M:M_1]}\fd_{M_2/L}^{[M:M_2]}
\supseteq m^{m^2}\fd_{M_1/L}^m.
\]
Inserting this into \eqref{discrelation-1}, noting that $[L:\Qq ]\leq n^2d$,
$[M:L]\leq m^2$,
we obtain
\begin{equation}\label{discrelation-2}
|D_M|\leq m^{m^2n^2d}(N_L\fd_{M_1/L})^m|D_L|^{m^2}.
\end{equation}

We estimate $N_L\fd_{M_1/L}$.
Let $\fP$ be a prime ideal of $O_L$ not dividing a prime ideal from $S$
and not dividing $ma_0bD(f)$. Then by Lemma \ref{L_gen_ell},
\[
\ord_{\fP}(\gamma_1\gamma_2^{-1})\equiv\ord_{\fP}
\left(\frac{a_0(x-\alpha_1)}{a_0(x-\alpha_2)}\right)\equiv 0\, ({\rm mod}\, m),
\]
and so by Lemma \ref{L_gen_disc_primes_excluded}, $M_1/L$ is unramified at $\fP$.
Consequently, $\fd_{M_1/L}$ is composed of prime ideals from $U$,
where $U$ is the set of prime ideals of $O_L$ that divide the prime
ideals from $S$ or $ma_0bD(f)$.
Using Lemma \ref{L_gen_disc_primedivisors}, it follows that
\begin{eqnarray}\label{discrelation-3}
\fd_{M_1/L}&\supseteq&\prod_{\fP\in U}\fP^{m(1+\ord_{\fP}(u(m))}
\\
\nonumber
&\supseteq&\prod_{\fP\in U}\fP^m\prod_{\fP}\fP^{m\ord_{\fP}(u(m))}
\supseteq u(m)^m\prod_{\fP\in U}\fP^m.
\end{eqnarray}
First, by prime number theory, $u(m)\leq m^{\pi (m)}\leq 4^m$
(see Rosser and Schoenfeld \cite[Corollary 1]{RosserSchoenfeld}).
Hence $|N_{L/\Qq}(u(m)^m)|\leq 4^{m^2n^2d}$.
Second, by an argument similar to the proof of \eqref{bound-3},
defining $V$ to be the set of prime ideals of $O_L$ which are contained
in $S$ or divide $ma_0bD(f)$,
\begin{eqnarray*}
N_L(\prod_{\fP\in U}\fP)&\leq& N_K(\prod_{\fp\in V}\fp)^{[L:K]}
\leq N_K(\prod_{\fp\in V}\fp)^{n^2}
\\
&\leq& (Q_SN_S(ma_0bD(f))^{n^2}\leq (Q_Se^{d\cdot h(ma_0bD(f))})^{n^2}
\\
&\leq& Q_S^{n^2}m^{n^2d}e^{2n^3d(\log n+\widehat{h})}
\leq Q_S^{n^2}m^{n^2d}n^{2n^3d}e^{2n^3d\widehat{h}}
\end{eqnarray*}
where in the last estimate we have used Lemma \ref{disc-height}.
By combining this estimate and that for $|N_{L/\Qq}(u(m)^m)|$ with \eqref{discrelation-3},
we obtain
\begin{equation}\label{discrelation-4}
N_L\fd_{M_1/L}\leq 6^{m^2n^2d}n^{2mn^3d}Q_S^{mn^2}e^{2mn^3d\widehat{h}}.
\end{equation}
Finally, by inserting this estimate and the one arising from
Lemma \ref{L_disc_I},
\begin{equation}\label{disc_1b}
|D_L| \leq n^{4n^2d}\cdot e^{4n^2d \widehat{h}}\cdot |D_K|^{n^2}
\end{equation}
into \eqref{discrelation-2}, after some computations, we obtain
\eqref{discestimate-1}.

We now prove (ii). Let $m=2$. Take $L=K(\alpha_1,\alpha_2,\alpha_3)$,
$M_1=L(\sqrt{\gamma_1/\gamma_2})$, $M_2=L(\sqrt{\gamma_1/\gamma_3})$,
so that $M=M_1M_2$. Completely similarly to \eqref{discrelation-4},
but now using $[L:K]\leq n^3$ instead of $\leq n^2$, we get
\[
N_L\fd_{M_1/L}\leq 6^{4n^3d}n^{4n^4d}Q_S^{2n^3}e^{4n^4d\widehat{h}}.
\]
For $N_L\fd_{M_2/L}$ we have the same estimate.
So by Lemma \ref{L_gen_disc_extensions} (ii),
\[
N_L\fd_{M/L}\leq (N_L\fd_{M_1/L})^2(N_L\fd_{M_2/L})^2\leq
6^{16n^3d}n^{16n^4d}Q_S^{8n^3}e^{16n^4d\widehat{h}}.
\]
By inserting this inequality and the one arising from
Lemma \ref{L_disc_I},
\[
|D_L| \leq n^{6n^3d}\cdot e^{6n^3d \widehat{h}}\cdot |D_K|^{n^3}
\]
into $|D_M|=N_L\fd_{M/L}|D_L|^{[M:K]}$, after some computations
we obtain \eqref{discestimate-2}.
\end{proof}

\begin{proof}[Proof of Theorem \ref{T_super}]
Let $m\geq 3$ and let $x,y\in O_S$ be a solution to $by^m=f(x)$ with $y\not= 0$.
We have $x-\alpha_i=\gamma_i\xi_i^m$ ($i=1\kdots n)$ with the $\gamma_i,\xi_i$
as in Lemma \ref{L_gen_ell}.
Let $M:=K(\alpha_1,\alpha_2,\root\m\of{\gamma_1/\gamma_2},\rho )$, where $\rho$ is a primitive $m$-th root of unity, and let $T$ be the set of places of $M$
lying above the places from $S$.
Let $\fp_1\kdots\fp_t$ be the prime ideals (finite places) in $S$,
and $\fP_1\kdots\fP_{t'}$ the prime ideals in $T$. Then
$t'\leq [M:K]t\leq m^2n^2t$.
Further, let $P_T:=\max_{i=1}^{t'} N_M\fP_i$,
$Q_T:=\prod_{i=1}^{t'} N_M\fP_i$.

We clearly have
\begin{equation}\label{super-thue}
\gamma_1\xi_1^m-\gamma_2\xi_2^m=\alpha_2-\alpha_1,\ \ \xi_1,\xi_2\in O_T,
\end{equation}
and the left-hand side is a binary form of non-zero discriminant
which splits into linear factors over $M$.
By
Proposition \ref{P_Thue}, we have
\begin{eqnarray}\label{super-bound}
&&h(\xi_1)\leq c_1'm^6P_TR_T\Big(1+\frac{\log^* R_T}{\log^* P_T}\Big)\times
\\
&&
\nonumber
\qquad\qquad\qquad
\times
\Big( R_M+h_M\cdot d_M^{-1}\log Q_T+md_MA+B),
\end{eqnarray}
where $A=\max (h(\gamma_1),h(\gamma_2)$, $B=h(\alpha_1-\alpha_2)$,
$d_M=[M:\Qq ]$ and
$c_1'$ is the constant $c_1$ from Proposition \ref{P_Thue},
but with $s,d$ replaced by the upper bounds $m^2n^2s$, $m^2n^2d$
for the cardinality of $T$ and $[M:\Qq ]$, respectively,
and $R_T$ is the $T$-regulator.

Using $d\leq 2s$ we can estimate $c_1'$
by the larger but less complicated bound,
\begin{equation}\label{super-bound-hulp-1}
c_1'\leq 2^{50}(4m^2n^2s)^{7m^2n^2s}.
\end{equation}
Next, by \eqref{bound-3},
\begin{equation}\label{super-bound-hulp-2}
P_T\leq Q_T\leq Q_S^{[M:K]}\leq Q_S^{m^2n^2}.
\end{equation}
Let $C$ be the upper bound for $|D_M|$ from \eqref{discestimate-1}.
Thus, by Lemma \ref{L_hR} and \eqref{3.4},
\[
\max (h_M,R_M)\leq 5C(\log^* C)^{m^2n^2d-1}.
\]
Further, $A$ can be estimated from above by the bound from \eqref{split_to_powers}, and $B$ by
\[
h(\alpha_1)+h(\alpha_2)+\log 2\leq h(f)+(n+1)\log 2\leq \widehat{h}+(n+1)\log 2
\]
in view of Lemma \ref{L_bound_roots_by_pol}.
Together with \eqref{super-bound-hulp-2}, this implies
\begin{eqnarray}\label{super-bound-hulp-3}
&&R_M+h_M\cdot d_M^{-1}\log Q_T+md_MA+B
\\
\nonumber
&&\qquad
\leq 7C(\log^* C)^{m^2n^2d-1}\cdot d^{-1}\log Q_S\leq 7C(\log^* C)^{m^2n^2d}.
\end{eqnarray}
Next, by \eqref{3.4}, the inequality $d+t\leq 2s$,
and \eqref{super-bound-hulp-2},
we have
\begin{eqnarray*}
R_T&\leq& C^{1/2}(\log^* C)^{m^2n^2d-1}(\log^* P_T)^{t'}
\\
&\leq& C^{1/2}(\log^* C)^{m^2n^2d-1}(m^2n^2\log^* Q_S)^{m^2n^2t}
\\
&\leq& (m^2n^2)^{m^2n^2s}C^{1/2}(\log^* C)^{2m^2n^2s-1}
\end{eqnarray*}
and
\[
1+\frac{\log^* R_T}{\log^* P_T}\leq
4m^2n^2s\log^* C,
\]
hence
\begin{equation}\label{super-bound-hulp-4}
P_TR_T\Big(1+\frac{\log^* R_T}{\log^* P_T}\Big)
\leq (4m^2n^2)^{m^2n^2s}Q_S^{m^2n^2}C^{1/2}(\log^* C)^{2m^2n^2s}.
\end{equation}
Combining \eqref{super-bound-hulp-3}, \eqref{super-bound-hulp-4}
with \eqref{super-bound} gives
\begin{eqnarray*}
h(\xi_1)&\leq& 7m^6 c_1'(4m^2n^2)^{m^2n^2s}Q_S^{m^2n^2}C(\log^* C)^{4m^2n^2s}
\\
&\leq& 2^{50}(4m^2n^2s)^{13m^2n^2s}Q_S^{m^2n^2}C^2.
\end{eqnarray*}
Using
\[
h(x)\leq \log 2+h(\alpha_1)+h(\gamma_1)+mh(\xi_1),\ \ \
h(y)\leq m^{-1}(h(b)+h(f)+nh(x)),
\]
and the upper bound for $h(\gamma_1)$ from \eqref{split_to_powers},
we get
\begin{equation}\label{super-bound-hulp-5}
h(x),h(y)\leq 2^{51}mn(4m^2n^2s)^{13m^2n^2s}Q_S^{m^2n^2}C^2.
\end{equation}
Now substituting $C$, i.e., the upper bound for $|D_M|$ from
\eqref{discestimate-1}, and some algebra gives
the upper bound \eqref{2.1} from Theorem \ref{T_super}.
\end{proof}

\begin{proof}[Proof of Theorem \ref{T_hyper}]
Let $x,y\in O_S$ be a solution to $by^2=f(x)$ with $y\not= 0$.
We have $x-\alpha_i=\gamma_i\xi_i^m$ ($i=1\kdots n)$ with the $\gamma_i,\xi_i$
as in Lemma \ref{L_gen_ell}.
Let
\[
M:=K(\alpha_1,\alpha_2,\alpha_3,\sqrt{\gamma_1/\gamma_3},\sqrt{\gamma_2/\gamma_3}),
\]
and let $T$ be the set of places of $M$
lying above the places from $S$.
Notice that $[M:K]\leq 4n^3$.
Then
\begin{equation}\label{hyper-pell}
\gamma_1\xi_1^2-\gamma_2\xi_2^2=\alpha_2-\alpha_1,\ \ \
\gamma_1\xi_1^2-\gamma_3\xi_3^2=\alpha_3-\alpha_1,
\ \ \xi_1,\xi_2\in O_T.
\end{equation}
By applying Proposition \ref{P_Pell} to \eqref{hyper-pell},
and doing the same computations as above, we obtain
the same bound as in \eqref{super-bound-hulp-5},
but with $m=2$ and $m^2n^2$ replaced by $4n^3$, and with $C$ the upper bound
for $|D_M|$ from \eqref{discestimate-2}.
After some computation, we obtain the bound \eqref{2.2}
from Theorem \ref{T_hyper}.
\end{proof}

\section{Proof of Theorem \ref{T_ST}}

We assume that in some finite extension $G$ of $K$, the polynomial
$f$ factorizes as $a_0(X-\alpha_1)\cdots (X-\alpha_n)$.
For $i=1\kdots n$, let $L_i=\Qq (\alpha_i)$, let $d_{L_i},h_{L_i},R_{L_i}$ denote
the degree, class number and regulator of $L_i$,
and let $T_i$ be the set of places of $L_i$
lying above the places in $S$. Further, denote by $R_{T_i}$ the $T_i$-regulator
of $L_i$, and denote by $t_i$ the cardinality of $T_i$.
Let $Q_{T_i}:=\prod_{\fP\in T_i} N_{L_i}\fP$, where the product is over all
prime ideals in $T_i$.
The group of $T_i$-units $O_{T_i^*}$ is finitely generated and
by Lemma 2 of \cite{GyYu1} (see also \cite{BugGy1}, \cite{BugGy2} and \cite{Bug1}) we may choose a fundamental system of $T_i$-units, i.e., basis of $O_{T_i}^*$ modulo torsion
 $\eta_{i1}, \dots, \eta_{i,t_i-1}$ such that
\begin{equation}\label{Sunit-estimates}
\left\{\begin{array}{l}
\displaystyle{\prod_{j=1}^{t_i-1} h(\eta_{ij})\leq c_{1i}R_{T_i}},
\\
\displaystyle{\max_{1\leq j\leq t_i-1} h(\eta_{ij}) \leq c_{2i}R_{T_i}},
\end{array}\right.
\end{equation}
where
\[
c_{1i}= \frac{((t_i-1)!)^2}{2^{t_i-2}d_L^{t_i-1}},\ \
c_{2i}=29e\sqrt{t_i-2}d_{L_i}^{t_i-1}\log^*d_{L_i}c_{i1}.
\]
We estimate these upper bounds from above.
First noting $t_i\leq [L_i:K]s\leq ns$
we have the generous estimate
\begin{equation}\label{constants}
c_{i1},c_{i2}\leq 1200t_i^{2t_i}\leq 1200 (ns)^{2ns}.
\end{equation}
For the class number and regulator $h_{L_i}$, $R_{L_i}$,
we have similarly to \eqref{bound-2}:
\begin{eqnarray}\label{h_LR_L-bound}
\max (h_{L_i},R_{L_i},h_{L_i}R_{L_i})
&\leq& 5|D_{L_i}|^{1/2}(\log^* |D_{L_i}|)^{nd-1}
\\
\nonumber
&\leq& (n^3d)^{nd}e^{(2n-2)d\widehat{h}}|D_K|^n.
\end{eqnarray}
Further, from \eqref{3.4}, $d\leq 2s$,
we deduce
\begin{eqnarray}\label{bound-Tregulator}
R_{T_i}&\leq& (n^3d)^{nd}e^{(2n-2)d\widehat{h}}|D_K|^n(\log^* P_{T_i})^{ns-1}
\\
\nonumber
&\leq& (n^3d)^{nd}e^{(2n-2)d\widehat{h}}|D_K|^n(n\log^* P_S)^{ns-1}
\\
\nonumber
&\leq& (4n^7s^2)^{ns}e^{(2n-2)d\widehat{h}}|D_K|^n(\log^* P_S)^{ns-1}.
\end{eqnarray}
By inserting this and \eqref{constants} into \eqref{Sunit-estimates}, we obtain
\begin{eqnarray}
\label{fund_Sunit_prod}
&&\prod_{j=1}^{t_i-1} h(\eta_{ij})\leq
C_1:= 1200(4n^9s^4)^{ns}e^{2nd\widehat{h}}|D_K|^n(\log^* P_S)^{ns-1},
\\
\label{fund_Sunit}
&&\max_{1\leq j\leq t_i-1} h(\eta_{ij}) \leq C_1.
\end{eqnarray}

Now let $x,y$ and $m$ satisfy
\begin{equation}\label{Schinzel-Tijdeman}
by^m=f(x),\ \ m\in\Zz_{\geq 3},\, x,y\in O_S,\,
y\not=0,\, y\ \mbox{not a root of unity},
\end{equation}

\begin{lemma}\label{L_powersplit-2}
For $i=1,2$ there are $\gamma_i,\xi_i\in L_i^*$,
and integers $b_{i1}\cdots b_{i,t_i}$ of absolute value at most $m/2$, such that
\begin{equation}\label{equation-powersplit_2}
\left\{\begin{array}{l}
(x-\alpha_i)^{h_{L_1}h_{L_2}}=
\eta_{i1}^{b_{i1}}\cdots\eta_{i,t_i-1}^{b_{i,t_i-1}}\gamma_i\xi_i^m,
\\[0.2cm]
h(\gamma_i)\leq C_2:=(2n^3s)^{6ns}|D_K|^{2n}e^{4nd\widehat{h}}(\widehat{h}+\log^*P_S).
\end{array}\right.
\end{equation}
\end{lemma}

\begin{proof}
For convenience, we put $r:=h_{L_1}h_{L_2}$.
By symmetry, it suffices to prove the lemma for $i=1$.
For notational convenience, in the proof of this lemma only,
we suppress the index $i=1$
(so $L=L_1,T=T_1,t=t_1$, etc.).
We use the same notation as in the proof of Lemma \ref{L_gen_ell}.
Similar to \eqref{g-sum}, \eqref{g-product}, we have
\[
\frac{[x-\alpha]}{[1,\alpha]}+\frac{[g(x)]}{[g]}\supseteq \frac{[D(f)]}{[f]^{2n-2}},
\ \ \
[b][f]^{-1}[y]^m=\frac{[x-\alpha]}{[1,\alpha]}\cdot \frac{[g(x)]}{[g]},
\]
where $[\cdot ]$ denote fractional ideals with respect to $O_T$.
From these relations, it follows that there are integral ideals
$\fB_1,\fB_2$ of $O_T$ and a fractional ideal $\fA$ of $O_T$, such that
\[
\frac{[x-\alpha ]}{[1,\alpha ]}=\fB_1\fB_2^{-1}\fA^m,
\]
where
\[
\fB_1\supseteq [b]\cdot \frac{[D(f)]}{[f]^{2n-2}},\ \
\fB_2\supseteq [f]\cdot \frac{[D(f)]}{[f]^{2n-2}}.
\]
Since
\[
[a_0][1,\alpha ]\subseteq [a_0]\prod_{j=1}^n [1,\alpha_j]
\subseteq [f]\subseteq [1],
\]
it follows that $[1,\alpha ]^{-1}\supseteq [a_0]$.
Hence
\[
[x-\alpha ]=\fC_1\fC_2^{-1}\fA^m,
\]
where $\fC_1,\fC_2$ are ideals of $O_T$ such that
\[
\fC_1,\fC_2\supseteq [a_0bD(f)].
\]
Raising to the power $r$, we get
\begin{equation}\label{final-power-relation}
(x-\alpha )^{r} =\gamma_1\gamma_2^{-1}\lambda^m,
\end{equation}
for some non-zero $\gamma_1,\gamma_2\in O_T$ and $\lambda\in L^*$ with
\[
[\gamma_k]\supseteq [a_0bD(f)]^r\ \ \mbox{for } k=1,2.
\]
By Lemma \ref{L_bound_hight_by_norm}, there exist $\varepsilon_1,\varepsilon_2\in O_T^*$
such that for $k=1,2$,
\[
h(\varepsilon_k\gamma_k)\leq \frac{r}{d_L}\log N_T(a_0bD(f))
+cR_L+\frac{h_L}{d_L}\log Q_T,
\]
where $c\leq 39d_L^{d_L+2}\leq 39(2ns)^{2ns+2}$.
There are $\varepsilon\in O_T^*$, a root of unity $\zeta$ of $L$,
and integers $b_1\kdots b_{t-1}$ of absolute value at most $m/2$, such that
\[
\varepsilon_2\varepsilon_1^{-1}
=\zeta \varepsilon^m\eta_1^{b_1}\cdots\eta_{t-1}^{b_{t-1}}.
\]
Writing
\[
\gamma := \zeta^{-1}\frac{\varepsilon_1\gamma_1}{\varepsilon_2\gamma_2},\ \ \
\xi :=\varepsilon\lambda
\]
where $\eta_1\kdots \eta_{t-1}$ are the fundamental units of $O_T^*$ satisfying
\eqref{fund_Sunit_prod}, \eqref{fund_Sunit},
we get
\[
x-\alpha = \eta_1^{b_1}\cdots\eta_{t-1}^{b_{t-1}}\gamma\xi^m,
\]
where
\begin{equation}\label{gamma-bound}
h(\gamma )\leq \frac{2r}{d_L}\log N_T(a_0bD(f))
+2cR_L+2\frac{h_L}{d_L}\log Q_T.
\end{equation}
By \eqref{h_LR_L-bound}, $d\leq 2s$, \eqref{a0bD(f)}, \eqref{bound-3} we have
\begin{eqnarray*}
&&h_L,R_L\leq (2n^3s)^{2ns}e^{2nd\widehat{h}}|D_K|^n,\ \
r=h_{L_1}h_{L_2}\leq (2n^3s)^{4ns}e^{4nd\widehat{h}}|D_K|^{2n},
\\
&&d_L^{-1}\log N_T(a_0bD(f))\leq (2n-1)\log n +2n\widehat{h},
\\
&&d_L^{-1}\log Q_T\leq d^{-1}\log Q_S\leq s\log^* P_S.
\end{eqnarray*}
By inserting these bounds into \eqref{gamma-bound} and using $n\geq 2$,
after some algebra we obtain the upper bound $C_2$.
\end{proof}

\begin{proof}[Completion of the proof of Theorem \ref{T_ST}]
In what follows, let $L:=K(\alpha_1,\alpha_2)$, $d_L:=[L:\Qq ]$,
$T$ the set of places of
$L$ lying above the places from $S$, and
$t$ the cardinality of $T$. Let again $x,y\in O_S$ and $m$ an integer $\geq 3$
with $by^m=f(x)$, $y\not= 0$ and $y$ not a root of unity.
Put
\[
X:=\max_{i=1\kdots n} h(x-\alpha_i).
\]
Without loss of generality we assume
\begin{equation}\label{m-lower-bound}
m\geq (10n^2s)^{38ns}|D_K|^{6n}P_S^{n^2}e^{11nd\widehat{h}}.
\end{equation}
Then
\begin{eqnarray}\label{lower-bound-special-case}
&&X\geq \max (C_3, m(4d)^{-1}(\log 3d)^{-3}),
\\
\nonumber
&&\qquad
\mbox{with } C_3:=
(10n^2s)^{37ns}|D_K|^{6n}P_S^{n^2}e^{11nd\widehat{h}}.
\end{eqnarray}
Indeed,
by Lemma \ref{L_lower bound_for hight} we have
\[
m\leq \frac{n\cdot X+h(a_0)+h(b)}{h(y)}\leq  (2d(\log (3d))^3(nX+2\widehat{h}).
\]
If $X<C_3$ this contradicts \eqref{m-lower-bound}.
If $X\geq C_3$ the other lower bound for $X$ in the maximum easily follows.

We assume without loss of generality, that
\[
X=h(x-\alpha_2).
\]
If $|x-\al_2|_v\leq 1$ for $v\in T$, then using $x\in O_S$ we have
$$
\begin{aligned}
X&\leq \frac{1}{d_L}\log \left( \prod_{v \not\in T} \max(1,|x-\al_2|_v)\right)\\
&\leq \frac{1}{d_L}\log \left( \prod_{v \not\in T} \max(1,|\al_2|_v)\right) \leq h(\al_2) \leq \frac{\log^*(n+1)}{2}+h(f),
\end{aligned}
$$
which is impossible by \eqref{lower-bound-special-case}.
Hence $\max_{v\in T} |x-\alpha_2|_v>1$.
Choose $v_0 \in T$ such that
\begin{equation}\label{max_val}
|x-\al_2|_{v_0}=\max_{v\in T} |x-\al_2|_v.
\end{equation}
Then we have
$$
\begin{aligned}
X&\leq \frac{1}{d_L}\left( \log \left( |x-\al_2|^t_{v_0}\prod_{v \not\in T} \max(1,|x-\al_2|_v)\right) \right)\\
&\leq \frac{1}{d_L}\left( \log \left( |x-\al_2|^t_{v_0}\prod_{v \not\in T} \max(1,|\al_2|_v)\right) \right).
\end{aligned}
$$
which gives
$$
|x-\al_2|_{v_0} \geq
\frac{e^{Xd_L/t}}{\prod_{v \not\in T} \max(1,|\al_2|_v)^{1/t}}.
$$
Thus we have
\begin{equation}\label{bound_interm_11a}
\begin{aligned}
\left| 1- \frac{x-\al_1}{x-\al_2} \right|_{v_0} &=
\frac{|\al_2-\al_1|_{v_0}}{|x-\al_2|_{v_0}}
&\leq  \frac{|\al_2-\al_1|_{v_0}\prod_{v \not\in T} \max(1,|\al_2|_v)^{1/t}}{e^{Xd_L/t}}.
\end{aligned}
\end{equation}
Put $s(v_0)=1$ if $v_0$ is real, $s(v_0)=2$ if $v$ is complex,
and $s(v_0)=0$ if $v_0$ is finite.
Since by Lemma \ref{L_bound_roots_by_pol} we have
$$
\begin{aligned}
|\al_2-\al_1|_{v_0}&\prod_{v \not\in T} \max(1,|\al_2|_v)^{1/t}
\\
&\leq 2^{s(v_0)}\max(1,|\al_2|_{v_0})\max(1,|\al_1|_{v_0})\prod_{v \not\in T} \max(1,|\al_2|_v)
\\
&\leq 2^{s(v_0)}\exp(d_L(h(\al_1)+h(\al_2)))
\\
&\leq 2^{(n+1)s(v_0)}\exp((d_Lh(f)),
\end{aligned}
$$
\eqref{bound_interm_11a} gives us
\begin{equation}\label{bound_interm_11}
\left| 1- \frac{x-\al_1}{x-\al_2} \right|_{v_0}\leq
\exp\Big( (n+1)s(v_0)\log 2\, +d_Lh(f)-Xd_L/t\Big).
\end{equation}
Notice that by \eqref{lower-bound-special-case} we have
\begin{equation}\label{bound_interm_12}
\left| 1- \frac{x-\al_1}{x-\al_2} \right|_{v_0} <1.
\end{equation}
In general, we have for $y\in L$ with $|1-y|_{v_0}<1$ and any
positive integer $r$,
\[
|1-y^r|_{v_0}\leq 2^{r\cdot s(v_0)}|1-y|_{v_0}.
\]
Hence
\[
\left| 1- \left(\frac{x-\al_1}{x-\al_2}\right)^{h_{L_1}h_{L_2}} \right|_{v_0}
\leq
\exp\Big( (h_{L_1}h_{L_2}+n+1)s(v_0)\log 2\, +d_Lh(f)-Xd_L/t\Big).
\]
Using \eqref{lower-bound-special-case} and the estimates
\eqref{h_LR_L-bound}, $h(f)\leq\widehat{h}$, $d_L\leq nd$, $s\leq t\leq ns$,
this can be simplified to
\begin{equation}\label{bound_upper_1}
\left| 1- \left(\frac{x-\al_1}{x-\al_2}\right)^{h_{L_1}h_{L_2}} \right|_{v_0}
\leq \exp (-Xd_L/2t).
\end{equation}

On the other hand using Proposition \ref{P_Baker} and
Lemma \ref{L_powersplit-2}
we get a Baker type lower bound
\begin{equation}\label{bound_lower_1}
\begin{aligned}
&\left| 1- \left(\frac{x-\al_1}{x-\al_2}\right)^{h_{L_1}h_{L_2}} \right|_{v_0}
\\
&\qquad =
\left| 1-  \frac{\ga_1}{\ga_2}
\cdot \eta_{11}^{b_{11}}\cdots \eta_{1,t_1-1}^{b_{1,t_1-1}}
\cdot \eta_{21}^{-b_{21}}\cdots \eta_{2,t_2-1}^{-b_{2,t_2-1}}
\cdot
\left(\frac{\xi_1}{\xi_2}\right)^m \right|_{v_0}
\\
&\qquad\geq \exp\Big(-c_1(t_1+t_2,d_L)\cdot
\frac{N(v_0)}{\log N(v_0)}\Theta \log B\Big)
\end{aligned}
\end{equation}
where
$$
\begin{aligned}\label{bound_parameters_in_lower}
&\Theta:=\max (h(\xi_1/\xi_2),m(d))\cdot\max (h(\gamma_1/\gamma_2),m(d))\cdot
\prod_{j=1}^{t_1-1}h(\eta_{1j})\cdot
\prod_{j=1}^{t_2-1}h(\eta_{2j}),
\\
&B:=\max\{3,m, |b_{11}|\kdots |b_{1,t_1-1}|,\, |b_{21}|\kdots |b_{2,t_2-1}|),
\\
&N(v_0):=\begin{cases} 2 & \text{if $v_0$ is infinite}
\\
N_L\fP & \text{if $v_0=\fP$ is a prime ideal $\fP$},\end{cases}
\\
&c_1(t_1+t_2, d_L):=12(16ed_L)^{3t_1+3t_2+2} (\log^* d_L)^2.
\end{aligned}
$$

We estimate the above parameters.
First, by
\eqref{equation-powersplit_2}, we have $h(\gamma_i)\leq C_2$ for $i=1,2$.
Moreover, the exponents $b_{ij}$ in \eqref{equation-powersplit_2}
have absolute values at most $m/2$.
Together with \eqref{fund_Sunit}
and \eqref{lower-bound-special-case},
these imply
\begin{eqnarray}\label{xi-bound}
h(\xi_1/\xi_2) &\leq&\max h(\xi_1)+h(\xi_2)
\\
\nonumber
&\leq&
\frac{2}{m}(X+C_2)+\mbox{$\frac{1}{2}$}(t_1+t_2-2)C_1
\leq \frac{3}{m}\cdot X+2nsC_1
\\
\nonumber
&\leq&
(3+4d(\log 3d)^3\cdot 2nsC_1)\cdot\frac{X}{m}\leq 4^{ns+2}C_1\cdot\frac{X}{m},
\end{eqnarray}
where we have used $t_1,t_2\leq ns$, $d\leq 2s$, $n\geq 2$.
Further, using \eqref{fund_Sunit_prod} and
$h(\gamma_1/\gamma_2)\leq 2C_2$,
we get
\begin{equation}\label{Theta-bound}
\Theta \leq C_1^2\cdot 4^{ns+2}C_1\cdot\frac{X}{m}\cdot 2C_2\leq C_4\cdot\frac{X}{m},
\end{equation}
where
\[
C_4:= 2\times 10^{7} \big( 4^{10}n^{45}s^{18}\big)^{ns}|D_K|^{5n}
e^{10nd\widehat{h}}(\widehat{h}+1)(\log^* P_S)^{3ns-2}.
\]
Next, using $d_L\leq n(n-1)d\leq 2n(n-1)s$, $t_1,t_2\leq ns$, we have
\begin{equation}\label{c1-bound}
c_1(t_1+t_2, d_L)\leq C_5:=(32en^2s)^{6ns+3}.
\end{equation}
Finally, by \eqref{bound-3}, \eqref{m-lower-bound} we have
\[
N(v_0)\leq P_T\leq P_S^{[L:K]}\leq P_S^{n(n-1)}
\]
and $B=m$ since the exponents $b_{ij}$ in \eqref{equation-powersplit_2}
have absolute values at most $m/2$.
Inserting these and \eqref{Theta-bound}, \eqref{c1-bound} into \eqref{bound_lower_1},
we arrive at the lower bound
\[
\left| 1- \left(\frac{x-\al_1}{x-\al_2}\right)^{h_{L_1}h_{L_2}} \right|_{v_0}
\geq\exp\Big( -C_4C_5P_S^{n(n-1)}\frac{X}{m}\log m\Big).
\]
A comparison with the upper bound \eqref{bound_upper_1} gives
\[
\exp\Big( -C_4C_5P_S^{n(n-1)}\frac{X}{m}\log m\Big)\leq\exp(-d_LX/2t).
\]
By dividing out $X$ and inserting $t\leq n^2s$, $d\leq 2s$, we arrive at
\[
\begin{aligned}
\frac{m}{\log m}&\leq 2n^2sC_4C_5P_S^{n(n-1)}
\\
&< (10n^2s)^{35ns}|D_K|^{5n}e^{10nd\widehat{h}}(\widehat{h}+1)\cdot
P_S^{n(n-1)}(\log^* P_S)^{3ns-1}.
\end{aligned}
\]
Applying the inequalities $(\log X)^B\leq (B/2\epsilon )^BX^{\epsilon}$
for $X>1$, $B>0$, $\epsilon >0$ and $X+1\leq (e^{c-1}/c)e^{cX}$ for $X>0$, $c\geq 1$,
we arrive at our final estimate
\[
m< (10n^2s)^{40ns}|D_K|^{6n}P_S^{n^2}e^{11nd\widehat{h}}.
\]
This completes our proof of Theorem \ref{T_ST}.
\end{proof}

\end{document}